\documentclass[a4paper]{amsart}
\usepackage{comment}

\usepackage{amsmath,amsfonts,amssymb,amsthm,bm,xcolor,yfonts,extpfeil,graphicx}

\usepackage[pdftex,linktocpage]{hyperref}
\usepackage{smartref}

\newcommand{\mypound}{\scalebox{0.8}{\raisebox{0.4ex}{\#}}}
\newcommand{\CC}{\mathbb C}
\newcommand{\QQ}{\mathbb Q}
\newcommand{\NN}{\mathbb N}
\newcommand{\ZZ}{\mathbb Z}

\newcommand{\kk}{\Bbbk}

\newcommand{\tensor}{\mathop{\otimes}}
\newcommand{\im}{\mathop{\mathrm{im}}}
\newcommand{\rank}{\mathop{\mathrm{Rank}}}  

\usepackage{setspace}
\newcommand{\marginparstretch}{0.6}
\let\oldmarginpar\marginpar
\renewcommand\marginpar[1]{\-\oldmarginpar[\framebox{\setstretch{\marginparstretch}\begin{minipage}{\marginparwidth}{\raggedleft\tiny #1}\end{minipage}}]{\framebox{\setstretch{\marginparstretch}\begin{minipage}{\marginparwidth}{\raggedright\tiny #1}\end{minipage}}}}

\theoremstyle{plain}
\newtheorem{theorem}{Theorem}
\newtheorem{proposition}[theorem]{Proposition}
\newtheorem{corollary}[theorem]{Corollary}
\newtheorem{lemma}[theorem]{Lemma}

\theoremstyle{definition}
\newtheorem{Definition}[theorem]{Definition}
\newtheorem{notation}[theorem]{Notation}

\newtheorem{example}[theorem]{Example}
\newtheorem{remark}[theorem]{Remark}

\DeclareMathOperator{\Aut}{Aut}
\DeclareMathOperator{\gr}{gr}

\DeclareMathOperator{\spec}{Spec}
\DeclareMathOperator{\mspec}{MaxSpec}
\DeclareMathOperator{\Hom}{Hom}
\DeclareMathOperator{\Tor}{Tor}

\DeclareMathOperator{\gdim}{gldim}
\DeclareMathOperator{\pdim}{pdim}

\DeclareMathOperator{\Deg}{deg}

\DeclareMathOperator{\Stab}{Stab}
\DeclareMathOperator{\lt}{lt}

\numberwithin{theorem}{section}        

\makeatletter                          
\let\c@equation\c@theorem              
\makeatother                           
\begin{document}
\title{Idealisers in skew Group Rings}
\author{Ruth Reynolds}
\date{\today}
\thanks{The University of Edinburgh, James Clerk Maxwell Building, Peter Guthrie Tait Road, Edinburgh EH9 3FD.
 { \tt r.a.e.reynolds@sms.ed.ac.uk}}

\maketitle

\begin{abstract}
  Let $C$ be a commutative noetherian domain, $G$ be a finitely generated abelian group which acts on $C$ and $B = C\mypound G$ be the skew group ring. For a prime ideal $I \lhd C$, we study the largest subring of $B$ in which the right ideal $IB$ becomes a two-sided ideal - the idealiser subring. We obtain necessary and sufficient conditions for when this idealiser subring is left and right noetherian. We also give an example of these conditions in practice which translates to an interesting number theoretic problem.
\end{abstract}
\tableofcontents
\section{Introduction}
Let $B$ be a ring, $I$ be a right ideal in $B$ and $\mathbb{I}_B(I) = \{b \in B \mid bI \subseteq I\}$. This subring of $B$ is called the {\em idealiser of $I$ in $B$} and is the largest subring of $B$ which contains $I$ as a two-sided ideal. There is no analogue of this construction in commutative rings and so this is an intrinsically noncommutative concept. 

Introduced by Ore in \cite{Ore}, idealisers are highly noncommutative rings with interesting and sometimes pathological behaviour. For example, Stafford used idealisers to construct a variety of left and right noetherian rings with peculiar ideal structure \cite{Sta}. Idealisers occur naturally in Artin and Stafford's classification of noncommutative projective curves \cite{AS95} so we expect idealisers to feature in any sufficiently general noncommutative classification. Hence a better understanding of the behaviour of these subrings is desired. In the Artin-Stafford classification idealisers come from noetherian graded rings which have the property that no Veronese is generated in degree $1$. This property never occurs for commutative rings. 

Idealisers are often good examples of rings that have different left and right structures. For example, if 
\begin{align*}B=\CC[x,y][t^{\pm};\sigma]\end{align*} is a skew Laurent ring where
\begin{align*}\sigma(x) = x+1 \textrm{ and } \sigma(y)=y,
\end{align*} and we consider the right ideal $(x,y)B$, then we show below that $\mathbb{I}_B((x,y)B)$ is right but not left noetherian. The reasons for this will become clear in the paper, but we mention that this depends on the dynamics of the orbit of $(x,y)$ under $\sigma$. A related example was essential in the proof by Sierra and Walton that the universal enveloping algebra of the Witt algebra is not noetherian \cite{SW} which answered a question which had been open for over $20$ years.

Idealisers were studied in detail by Robson in \cite{Robson1972} where he observed that in the case that right ideal $I$ is semimaximal, that is to say, $I$ is the intersection of maximal right ideals, then the properties of the $B$ and $\mathbb{I}_B(I)$ are very closely linked. In particular, the following result characterises the noetherianity of an idealiser of a maximal right ideal. 
\begin{theorem}\cite[Theorem 2.3]{Robson1972}\label{rob}
Let $B$ be a ring, $I$ a maximal right ideal of $B$. Then $B$ is right noetherian if and only if $\mathbb{I}_B(I)$ is right noetherian.
\end{theorem} 
We note that Robson does not address left noetherianity at all in his paper and, as a consequence of results in this paper, a left noetherian version of Theorem \ref{rob} in a graded setting does not hold. Indeed, the idealiser mentioned above is a counterexample.

The objective of this paper is to investigate the noetherianity of idealisers in skew group rings. In \cite{Rog} and \cite{Sierra}, Rogalski and Sierra considered what it means for idealisers in twisted homogeneous coordinate rings \cite{AV} to be left and right noetherian and they obtained the following result which we paraphrase.

\begin{theorem}\cite[Theorem 10.2]{Sierra}\label{vague} Let $B = B(X, \mathcal{L}, \sigma)$ be a twisted homogeneous coordinate ring, where  $X$ is a projective variety, $\mathcal{L}$ is an appropriately ample invertible sheaf on $X$, $\sigma \in \Aut X$, and let $I$ be a right ideal of $B$ corresponding to a closed subscheme $Z \subseteq X$ of infinite order under $\sigma$. Then the noetherianity of $\mathbb{I}_B(I)$ is determined by the orbit of $Z$ under $\sigma$.
\end{theorem}
We state this result precisely in Theorem \ref{notvague} but, whilst we are being vague about the technicalities, it is interesting to note that an algebraic condition which is sometimes very difficult to verify, that of noetherianity, in this case can be understood from a geometric perspective. 

Let $C$ be a commutative domain and let $\ZZ$ act on $C$ by powers of $\sigma \in \Aut C$. Based on the behaviour of idealisers in twisted homogeneous coordinate rings, one naturally conjectures that similar geometric conditions to those in Theorem \ref{vague} control noetherianity of idealisers in $C\mypound\ZZ$ and indeed we verify that this is the case. However, it is difficult to predict what will control noetherianity of idealisers in $C\mypound G$ for an arbitrary group $G$. In this paper we completely answer the question for $G$ being a finitely generated abelian group.

\begin{Definition}
Let $I,J\lhd C$, a commutative ring. Then we say that $V(I)$ and $V(J)$ are {\em homologically transverse} if 
$$\Tor_j^C(C/I,C/J)=0$$
for all $j\geq 1$.
Further, for a group action of $G$ on $C$, we say $\{V(I^g)\}_{g \in G}$ is {\em critically transverse} if for all ideals $K\lhd C$,  $V(I^g)$ and $V(K)$ are homologically transverse for but finitely many $g \in G$.
\end{Definition}
Our main results are as follows:
\begin{theorem}[Theorem \ref{R}]\label{T1}
Let $C$ be a commutative noetherian domain, let $G$ be a finitely generated abelian group acting on $C$ and let $B = C\mypound G$. Let $I\lhd C$ be a prime ideal with trivial stabiliser under the $G$-action on $C$. Then the following are equivalent:
\begin{enumerate}
    \item[(1)] for all points $p \in V(I)$ the set
    $$\{h\in G \mid p^h \in V(I)\}$$
    is finite;
    \item[(2)] $\mathbb{I}_B(IB)$ is right noetherian.
\end{enumerate}
\end{theorem}
\begin{theorem}[Theorem \ref{finalLeft}]\label{T2}
With the same setup as Theorem \ref{T1} the following are equivalent:
\begin{enumerate}
    \item[(1)] $\{V(I^g)\}_{g \in G}$ is critically transverse;
    \item[(2)] $\mathbb{I}_B(IB)$ is left noetherian.
\end{enumerate}
\end{theorem}
We give a complete characterisation of noetherianity for idealisers at prime ideals of $C$ with a more general group action in the body of the paper.

Understanding when the conditions in Theorems \ref{T1} and \ref{T2} hold can be subtle. For example, let $C = \CC[x,y]$ and let $\ZZ^2$ act on $\CC[x,y]$ by translation. Then $\mathbb{I}_B((x-7y^2-1)B)$ is neither left nor right noetherian. However, if $f\in \CC[x,y]$ defines either an irreducible curve of genus $\geq 1$ or a line, then $\mathbb{I}_B(IB)$ is left and right noetherian. Full details of this example may be found in Section $5$.

In the first section and second section, we focus on group rings of polycyclic-by-finite groups and that working with graded rings simplifies results from the literature. In the fourth section we prove Theorems \ref{T1} and \ref{T2} and in the final section we give the details of the example mentioned above which raises an interesting question in number theory.

\textbf{Acknowledgements.}
The author is an EPSRC-funded student at the University of Edinburgh,
and the material contained in this paper will form part of her PhD thesis. The author would like to thank her supervisor Susan J. Sierra for suggesting this problem and providing guidance, and also to the EPSRC.

\section{Preliminaries}
The aim of this paper is to generalise theorems in the literature about idealisers in twisted homogeneous coordinate rings to some other well-chosen noncommutative situations.  In this section we give definitions and results which we will need. Our idealisers will be group-graded, and we begin by discussing the noetherianity of group-graded rings.

Chin and Quinn \cite{CQ} show that if $G$ is a polycyclic-by-finite group and $R$ is a $G$-graded ring, then all $G$-graded right ideals of $R$ are finitely generated if and only if $R$ is right noetherian. We note that it is still an open question as to whether this holds for rings graded by arbitrary groups. Indeed, it is even still a question as to whether $R\mypound G$, the group ring of $G$, can be noetherian when $G$ is not polycyclic-by-finite. Hence it is clearly not reasonable for us to consider the noetherianity of rings graded by non-polycyclic-by-finite groups.  

The proof by Chin and Quinn \cite{CQ} is rather inexplicit, so we begin with a direct proof in the case that $G$ is a finitely generated abelian group. This generalises a reult of Bj\"ork \cite[Theorem 2.18]{Bjork}.
\begin{proposition}
Let $G$ be a finitely generated abelian group and let $R$ be a $G$-graded ring. If all homogeneous right (left) ideals of $R$ are finitely generated then $R$ is right (left) noetherian.
\end{proposition}
\begin{proof}
We proceed by induction on $\rank(G) = n$.
Suppose initially that $n=0$, then $G = \{0,g_1\,\dots,g_m\}$ is finite and $R = R_0\oplus \left(\bigoplus_{i=1}^mR_{g_i}\right)$. We show that each $R_g$ is finitely generated as a right $R_0$-module. Indeed, let $I_1 \subsetneq I_2 \subsetneq \dots$ be a strictly ascending chain of right $R_0$-submodules in $R_g$, then $I_1R\subseteq I_2R \subseteq\dots$ is an ascending chain of $G$-graded right ideals in $R$. Further, since $(I_jR)_g = I_j$ this chain in strictly ascending and so must stabilise, hence $R_g$ is a noetherian right $R_0$-module. Thus, as $G$ is a finite group, $R$ is finitely generated as a right $R_0$-module and hence is right noetherian.

So we have proved the case when $n=0$. Now we suppose that $n>0$. Then $G$ contains a normal subgroup isomorphic to $\ZZ$; abusing notation we write $\ZZ \lhd G$. First we show that $R$ is $G/\ZZ$-graded. For an element $x \in G$, let $\bar{x}$ denote its image in $G/\ZZ$ and let $R_{\bar{x}} = \oplus_{a\in \ZZ}R_{x+a}$. Let $\bar{g},\bar{h}\in G/\ZZ$. Then,
\begin{align*}
R_{\bar{g}}R_{\bar{h}} &= \left(\bigoplus_{a\in \ZZ}R_{g+a} \right)\left(\bigoplus_{b\in \ZZ}R_{h+b} \right) =  \sum_{a,b\in \ZZ}R_{g+a}R_{h+b} \\ &\subseteq\sum_{a,b \in \ZZ}R_{g+a+h+b} \\
&= \sum_{a,b \in \ZZ}R_{(g+h)+a+b} = R_{\bar{g+h}}, \quad \textrm{ as needed.}
\end{align*}
We next prove the claim:

if $R$ has ACC on $G$-graded right ideals then that $R$ has ACC on $G/\ZZ$-graded ideals.
\noindent
We use the general method of Bj\"ork's proof \cite[Theorem 2.18]{Bjork}. Let $L \leq R_R$ be a $G/\ZZ$-graded right ideal and let $X \subseteq G$ be a set of coset representatives for $G/\ZZ$ so then 
$L  = \bigoplus_{x \in X}L_{\bar{x}}$. Let $t$ be an indeterminate. We begin by constructing the {\em external homogenisation} of a $G/\ZZ$-homogeneous element $z \in L$. We have $z\in L_{\bar{g}}$ for some $g \in X$, and we write $z = \sum_{h\in \ZZ}z_{g+h}$. Let $v = v(z) = \max \{h \in \ZZ\mid z_{g+h}\neq 0\}$ and define $z^* = \sum_{h \in \ZZ}z_{g+h}t^{v-h}\in R[t]$. Then define $L^* = \left<z^*\mid z \in L \right>$ as a right ideal of $R[t]$. We note that $L^*$ is $G$-graded, under the grading on $R[t]$ where $R$ is $G$-graded and $t$ is given degree $1\in \ZZ \leq G$. Under this grading, for each $z \in L_{\bar{g}}$, $z^*$ will be $G$-homogeneous of degree $g+v(z)$ and hence $L^*$ will be $G$-graded as it is generated by these elements.

There is a positive filtration $F_{\bullet}$ on $R[t]$ with $F_w = R+Rt+\dots +Rt^w$. Now consider the image $\sigma(L^*)$ in $\gr_{F}(R[t])$. Then $\sigma(L^*) = \bigoplus_{w \in \NN}(L^*\cap F_{w})/(L^*\cap F_{w-1}) \cong \bigoplus_{w \in \NN} J(w)t^w$. We observe that the $J(w)$ are right ideals in $R$ such that $J(w)\subseteq J(w+1)$. Further, these right ideals $J(w)$ are $G$-graded. Indeed, by virtue of coming from $L$, the $J(w)$ are $G/\ZZ$-graded. Now suppose $h = \sum_{i=1}^nh_i \in J(w)$ such that $h_i\in R_{g+a_i}$ where the $a_i\in \ZZ$ are distinct. Then $$\left(\sum_{i=1}^nh_i\right)t^w+[\textrm{lower powers of } t ]\in L^*.$$

Recall that $L^*$ is $G$-graded and, letting $N_i = g+a_i+w$, we see
$$\left(\left(\sum_{i=1}^nh_i\right)t^w+[\textrm{lower powers of }t]\right)_{N_i} = h_it^w+[\textrm{lower powers of } t ]_{N_i}\in L^*$$
and hence 
$$\sigma\left(h_it^w+\left([\textrm{lower powers of }t ]\right)_{N_i}\right) = h_it^w \in \sigma(L^*).$$

Thus we have an ascending chain of $G$-graded right ideals of $R$, $J(0) \subseteq J(1) \dots $, which stabilises to $J_{\infty}$ by assumption. 

We claim this is enough to show that $\sigma(L^*)$ is finitely generated. Indeed we reproduce the standard argument from Hilbert's basis theorem. Let $a_1,\dots, a_n$ be a finite generating set for $J_{\infty}$ and let $f_i \in \sigma(L^*)$ such that $\lt(f_i) = a_i$ where $\lt(g)$ is the coefficient of the highest power of $t$ in $g$. Without loss of generality we may assume $\Deg_t(f_i) = m$ for all $i$ (else, if $m'$ is the maximum of the $t$-degrees of the $f_i$ then we may replace any $f$ of lower $t$-degree with $f_it^{m'-\Deg_t(f_i)}$ Let $L_0 = \left(\sigma(L^*) \cap \sum_{i=0}^{m-1}t^iR[t]\right)+\sum_{i=1}^nf_iR[t]$, then we claim $L_0 = \sigma(L^*)$. Suppose not, obviously $L_0 \subseteq \sigma(L^*)$ so let $f \in \sigma(L^*)\setminus{L_0}$ be of minimal $t$-degree $d \geq m$. Since $f \in \sigma(L^*)$, $\lt(f) \in J_{\infty}$ and as such $\lt(f) = \sum_{i=1}^na_ir_i$ for some $r_i \in R$. Consider $g = f-\sum_{i=1}^nf_ir_it^{d-m}$, then this cancels out the leading term of $f$ and since $\Deg_t(g)<\Deg_t(f)$, by minimality $f = \sum_{i=1}^nf_ir_it^{d-m} \in L_0$, as required. Hence $\sigma(L^*)$ is finitely generated. Since $F$ is a positive filtration, by \cite[Proposition 2.11]{Bjork}, $L^*$ is finitely generated. 

We now define $\phi: R[t] \to R$ by $\phi (\sum_w t^wx_w) = \sum_w x_w$ (i.e. mapping $t$ to $1$). This is a surjective ring homomorphism and $L = \phi(L^*)$, hence $L$ is finitely generated.

So through this we have proved that $R$ being right $G$-graded-noetherian implies that $R$ is right $G/\ZZ$-graded noetherian. As $\rank(G) > \rank(G/\ZZ)$, by induction on $\rank(G)$ we conclude that $R$ is right noetherian.
\end{proof}
We have the following useful lemma.
\begin{lemma}\label{noeth}
Let $A\subseteq A'$ be subrings of a domain $D$. Suppose that $A$ is right noetherian and that $A$ contains a non-zero right ideal of $D$. Then $A'$ is right noetherian.
\end{lemma}
\begin{proof}
Let $0 \neq J \leq D_D$ such that $J \subseteq A$ and let $0 \neq y \in J$. Then we observe that $yA' \subseteq JD \subseteq A$, so $yA'$ is a right ideal of $A$, and hence is a finitely generated right $A$-module. As $D$ is a domain, $A'\cong yA'$ as  right $A$-modules and hence $A'$ is a finitely generated $A$-module and so is right noetherian.
\end{proof}
Now we give some general results from ring theory which will be used for the proofs of Theorems \ref{T1} and \ref{T2}. The starting point for the treatment of both left and right noetherianity is to generalise results of Rogalski and Stafford to the $G$-graded setting where $G$ is a polycyclic-by-finite group. Note that the question of whether the noetherianity of $G$-graded rings is completely controlled by homogeneous ideals is still open for an arbitrary group $G$. Hence we only consider rings graded by polycyclic-by-finite groups as those are the only ones whose noetherianity may be confirmed by checking homogeneous ideals. 

\begin{lemma}(cf.\cite[Lemma 1.1]{Sta})\label{B/I}
Let $G$ be a polycyclic-by-finite group and let $B$ be a right noetherian $G$-graded ring. Let $I \leq_\textrm{gr} B_B$ and let $R=\mathbb{I}_B(I)$. Suppose further that $B/I$ is a graded-noetherian right $R$-module. Then $R$ is a right noetherian ring.
\end{lemma}
\begin{proof}
This is a graded version of \cite[Lemma 1.1]{Sta}.
Let $J \leq_{\mathrm{gr}}R_R$. Then $JB$ and $JI$ are graded right ideals of $B$, thus are finitely generated, and $JB \geq JR = J \geq JI $. We can write $JB = \sum_{i=1}^ns_iB$ and $JI = \sum_{j=1}^m r_j I$ where $s_i, r_j \in J$. We note that we may actually assume $\{s_i\}=\{r_j\}$. Indeed, let $\{t_k\} = \{s_i\} \cup \{r_j\}$. Then $\{t_k\}$ generates $JB$, but we also note that $JI = \sum_j r_j I \subseteq \sum_k t_k I \subseteq JI$ so the $\{t_k\}$ also generate $JI$. So without loss of generality $m=n$ and $s_i = r_i$ for all $i = 1,\dots, n$.

Then we have a surjection
\[(B/I)^n \xtwoheadrightarrow{(s_1, \dots, s_n)} JB/JI \supseteq J/JI.\]

Hence, since $B/I$ is graded-noetherian, $J/JI$ is a finitely generated right $R$-module. Say $J/JI = \sum_{\ell=1}^p(u_\ell +JI)R$ for some $u_{\ell}\in J$ . Then observe that $J = \sum_\ell u_\ell R + \sum_i s_i I \subseteq  \sum_\ell u_\ell R + \sum_i s_i R \subseteq J$, so $J$ is finitely generated. Hence, by \cite{CQ}, $R$ is right noetherian.
\end{proof}
\begin{Definition}\label{quot}
Let $I,J$ be right ideals of a ring $B$. Define
$$(J:I) = \{b \in B \mid bI\subseteq J\}$$
to be the {\em ideal quotient}.
\end{Definition}
We alert the reader that this is a symmetric notation for an asymmetric concept. We will not use the corresponding left-handed version. We also observe that if $I$ and $J$ are graded right ideals of $B$, then $(J:I)$ will be graded as well.
\begin{proposition}(cf.\cite[Proposition 2.1]{Rog})\label{right-hom} Let $G$ be a polycyclic-by-finite group and let $I$ be a graded right ideal of a $G$-graded right noetherian ring $B$ which is a domain. Let $R = \mathbb{I}_B(I)$. Then the following are equivalent:
\begin{enumerate}
    \item[(1)] $R$ is right noetherian.
    \item[(2)] For every graded right ideal $J\supseteq I$ of $B$, $\Hom_B(B/I,B/J)$ is a right graded-noetherian $R$-module (or $R/I$-module).
\end{enumerate} 
\end{proposition}
\begin{proof}
This is a graded version of \cite[Proposition 2.1]{Rog} noting that the two-sided noetherien hypothesis on the overring $B$ in that result is superfluous.
Suppose $(2)$ holds.  By Lemma \ref{B/I} it suffices to show that $B/I$ is a graded-noetherian right $R$-module. By hypothesis, $R/I \cong \Hom_B(B/I,B/I)$ is right graded-noetherian and so it suffices to show $B/R$ is right graded-noetherian. To this end, let $K_R \leq_{\mathrm{gr}} B_R$ such that $R_R \leq K_R$. 
As $B$ is noetherian, let $J = KI = \sum_{i=1}^n k_i I$. 
As $R \leq K$ we have that $I \subseteq J$. Let $C = (J:I) = \{b \in B \mid bI \subseteq J\}\leq_{\gr}B_R$. Since $KI\subseteq J$, $K \subseteq C = C+J$.

We have the following identification: $\Hom _B(B/I,B/J) \cong (J:I)/J  = C/J$, thus by hypothesis $C/J$ is a right graded-noetherian $R$-module. 
Hence the submodule $K/J$ is finitely generated, say $K/J = \sum_{\ell=1}^m (a_{\ell}+J)R$ where $a_{\ell}\in K$ and so $K = \sum_{\ell=1}^m a_{\ell}R + \sum_{i=1}^n k_i I \subseteq \sum_{\ell=1}^m a_{\ell} R + \sum_{i=1}^n k_i R \subseteq K$ is finitely generated as required. Thus, as $R$ is $G$-graded, $R$ is right noetherian.

Conversely, observing that $R\subseteq B$ which contains a right ideal of $B$, namely $I$, we use Lemma \ref{noeth} to conclude $B_R$ is finitely generated and the rest of the argument may be found in  \cite[Proposition 2.1]{Rog}.
\end{proof}

\begin{proposition}(cf. \cite[Proposition 2.2]{Rog} \label{Toriff})
Let $G$ be a polycyclic-by-finite group and let $R$ be the idealiser of a graded right ideal $I$ in a left noetherian $G$-graded ring $B$ . Then the following are equivalent:
\begin{enumerate}
    \item[(1)] $R$ is left noetherian;
    \item[(2)] $\frac{BJ\cap R}{J}$ is a left graded-noetherian $R$-module (or $R/I$-module) for all finitely generated $J\leq_{\gr }$\({ }_R R\);
    \item[(3)] $R/I$ is left graded-noetherian and $\Tor_1^B(B/I,B/K) = (I\cap K)/IK$ is a graded-noetherian left $R$-module ($R/I$-module) for all $K\leq_{\gr}$\( _B B\).
\end{enumerate}
\end{proposition}
\begin{proof}
The equivalence of statements $(1)$ and $(2)$ may be found in \cite[Lemma 5.10]{R-generic}, and the argument $(1)$ implies $(3)$ may be found in \cite[Proposition 2.2]{Rog}. We note that these statements are not in the $G$-graded setting but the proofs follow in the same way.

We now show $(3)$ implies $(2)$. Let $J$ be an arbitrary finitely generated graded left ideal of $R$ and consider the following short exact sequences of left $R/I$-modules:

$$0 \to \frac{J}{IBJ} \to \frac{BJ \cap R}{IBJ}\to \frac{BJ\cap R}{J} \to 0$$

and

$$0 \to \frac{BJ\cap I}{IBJ} \to \frac{BJ \cap R}{IBJ}\to \frac{BJ\cap R}{BJ\cap I} \to 0.$$

Note $\Tor_1^B(B/I,B/BJ) \cong \frac{I\cap BJ}{IBJ}$ and $\frac{BJ\cap R}{BJ\cap I}$ injects into $R/I$. Thus, by assumption, the outer terms of the second short exact sequence are graded-noetherian left $R/I$-modules, so $\frac{BJ \cap R}{IBJ}$ is a graded-noetherian left $R/I$-module. Hence, from the first short exact sequence $\frac{BJ\cap R}{J}$ is a left graded-noetherian  $R/I$-module as required.
\end{proof}

\section{Idealisers in skew group rings}
We have shown that the noetherianity of idealisers in left and right noetherian domains graded by polycyclic-by-finite groups is completely determined by the properties of $\Hom$ and $\Tor$ spaces associated to homogeneous ideals. Let us now turn to the situation of interest to us, when $B$ is the skew group ring of a polycyclic-by-finite group, and prove some further useful results. We begin by fixing notation.
\begin{Definition}
Given a ring $R$ and a group $G$ which acts on $R$, for $g\in G$, $r\in R$ the image of $r$ under action by $g$ will be denoted $r^g$. We denote by $R\mypound G$ the skew group ring which is a free left $R$-module with elements of $G$ as a basis and with multiplication determined by
$$(rh)(sg) = (rs^h)(hg)$$
for $g,h \in G$ and $r,s \in R$. Each element of $R\mypound G$ can be written uniquely as $\sum_{g\in G}r_g g$ with $r_g=0$ for all but finitely many $g \in G$.
\end{Definition}

Many results in \cite{Sierra} are proved under the assumption that the subscheme $Z$ at which one idealises has infinite order under $\sigma$, that is $\Stab_{\sigma}(Z)$ is trivial. We wish to allow nontrivial stabilisers, which will require some notation, which will be in force for the remainder of the paper.

From here on $\kk$ will denote an algebraically closed field.
\begin{notation}\label{Not}We let $I$ be a prime ideal in a commutative noetherian domain $C$ which is a $\kk$-algebra. Let $G$ be a polycyclic-by-finite group which acts on $C$ and let $B \coloneqq C\mypound G$ be the skew group ring which is left and right noetherian by Theorem \ref{MR} . Let $R=\mathbb{I}_B(IB)$ denote the idealiser in $B$ of the right ideal $IB$. Since $IB$ is graded, so is $R$.
\end{notation}
We have the following result about the noetherianity of skew group rings from McConnell and Robson.
\begin{theorem}\label{MR}\cite[1.5.12]{MR}
Let $G$ be a polycyclic-by-finite group and $R$ a ring. If $R$ is right (left) noetherian then $R\mypound G$ is right (left) noetherian.
\end{theorem}
As noted earlier, it is still an open question as to whether $R\mypound G$ being noetherian implies that $G$ is a polycyclic-by-finite group. Thus we not consider groups more general than polycyclic-by-finite as we do not know whether $C\mypound G$ will be noetherian, and so all of the results from Section $2$ no longer hold.
Now that we have specified the rings with which we are working, we can be more precise about the structure of $\mathbb{I}_B(IB)$.
\begin{lemma}
\label{idealiser}
Assume Notation \ref{Not} and recall the notation $(J:I)$ from Definition \ref{quot}. Let $J\lhd C$ and consider $JB\leq B_B$. Then $$(JB:IB) = \bigoplus_{g \in G}(J:I^g)g.$$ Further $$R  = \bigoplus_{g\in G}(I:I^{g}) g.$$
\end{lemma}
\begin{proof}
We note that as both $JB$ and $IB$ are graded, then $(JB:IB)$ is also graded. Then for $g \in G$ we have the following identifications:
\begin{align*}
(JB:IB)_g &= \{c \in C \mid cgIB\subseteq JB\}g\\
          &= \{c \in C \mid cI^gB \subseteq JB\}g\\
          &= \{c\in C \mid cI^g \subseteq J\}g = (J:I^g)g.
\end{align*}
Observing $R= \{b\in B\mid bIB\subseteq IB\} = (IB:IB)$ gives the result.

\end{proof}
We note that if $I$ has trivial stabiliser under the action by $G$, then $R= C+IB$ by the primeness of $I$. 
We also have the following reductions of Propositions \ref{right-hom} and \ref{Toriff} using that $B$ is strongly graded.
\begin{theorem}\label{tor}
Assume Notation \ref{Not}. Then the following are equivalent:
\begin{enumerate}
    \item [(1)]$\bigoplus_{g \in G}\Tor_1^C(C/I,C/P^g)g\}$ is a finitely generated left $R$-module for all prime ideals $P \lhd C$;
    \item[(2)] $R$ is left noetherian.
\end{enumerate}
\end{theorem}
\begin{proof}
By Proposition \ref{Toriff}, $R$ is left noetherian if and only if $\Tor_1^B(B/IB,B/K) = (IB\cap K)/IBK$ is a finitely generated left $R$-module for all graded left ideals $K\leq B$. Every graded left ideal of $B$ is of the form $K = BJ$ where $J\lhd C$ as $B$ is strongly graded. Hence $R$ is left noetherian if and only if $\Tor_1^B(B/IB,B/BJ)$ is a finitely generated left $R$-module for all ideals $J\lhd C$.

We now prove that the following statements are equivalent:
\begin{enumerate}
    \item[(a)] $\Tor_1^B(B/IB,B/BJ)$ is a finitely generated left $R$-module for all ideals $J\lhd C$;
    \item[(b)] $\Tor_1^B(B/IB,B/BP)$ is a finitely generated left $R$-module for all prime ideals $P\lhd C$;
\end{enumerate}
That $(\mathrm{a})$ implies $(\mathrm{b})$ is clear. 

We have the following result from commutative algebra \cite[Proposition 3.7]{Eis}:

There exist $C$-modules $M_i$ and prime ideals $P_i\lhd C$ for $i=0,\dots, n$ such that
$$0=M_0\subset M_1\subset \dots \subset M_{n-1}\subset M_n = C/J$$
and $M_{j+1}/M_j \cong C/P_{j+1}$ for $j=0,\dots, n-1$.

Assume $(\mathrm{b})$. We show by induction on $j$ that $\Tor_1^B(B/IB,B\otimes_C M_j)$ is a finitely generated left $R$-module. The statement is trivially true for $j=0$ as then $B\otimes_C M_0\cong 0$.  Let us now consider the short exact sequence
\begin{align*}
\begin{array}{ccccccccc}
  0 &\longrightarrow & M_{j-1} & \longrightarrow & M_j & \longrightarrow & C/P_{j} &  \longrightarrow & 0.
\end{array}
\end{align*}
Applying $B\tensor_C -$ which is exact as $B_C$ is flat, we obtain
\begin{align*}
\begin{array}{ccccccccc}
  0 &\longrightarrow & B \otimes_C M_{j-1} & \longrightarrow & B\otimes_C M_j & \longrightarrow & B/BP_{j} &  \longrightarrow & 0.
\end{array}
\end{align*}
Applying $B/IB \otimes_B -$ gives rise to a long exact sequence which contains the following terms
\[\dots \to\Tor_1^B(B/IB, B \otimes_C M_{j-1}) \xrightarrow{\alpha}  \Tor_1^B(B/IB, B \otimes_C M_{j})  \xrightarrow{\beta}  \Tor_1^B(B/IB, B/BP_{j})   \to  \dots,\]
from which we extract the short exact sequence
\[0\to \ker \beta \to \Tor_1^B(B/IB, B \otimes_C M_{j}) \xrightarrow{\beta} \im\beta \to 0.\]
Since $\ker \beta \cong \im \alpha$ is a homomorphic image of $\Tor_1^B(B/IB, B \otimes_C M_{j-1})$, which is finitely generated by induction, and $\im \beta$ is a submodule of $\Tor_1^B(B/IB, B/BP_{j})$, which is finitely generated by assumption, we conclude $\Tor_1^B(B/IB, B \otimes_C M_{j})$ is a finitely generated left $R$-module as required. Hence $(\mathrm{a})$ holds.

Now that we have shown the equivalence of these statements we have the following identifications of left $R$-modules:
\begin{align*}
    \Tor_1^B(B/IB,B/BP) &= (IB\cap BP)/IPJ = \bigoplus_{g \in G}\frac{I\cap P^g}{IP^g} g\\
    &=\bigoplus_{g \in G}\Tor_1^C(C/I,C/P^g)g
\end{align*}
which completes the proof.
\end{proof}

We also have a similar result for right noetherianity:
\begin{theorem}\label{STPrime}
Assume Notation \ref{Not}. Then the following are equivalent:
\begin{enumerate}
    \item [(1)]$\bigoplus_{g \in G}\frac{(P:I^g)}{P}g$ is a finitely generated right $R$-module for all prime ideals $P\lhd C$ which contain $I$;
    \item[(2)] $R$ is right noetherian.
\end{enumerate}
\end{theorem}
We omit the proof from Proposition \ref{right-hom} as this follows in exactly the same style as Theorem  \ref{tor} follows from Propositin \ref{Toriff}.
\section{Idealisers in skew group rings of abelian groups}
\subsection{Right noetherianity}
Let us first consider right noetherianity for the setup of Notation \ref{Not}. By Theorem  \ref{STPrime}, we must show $\Hom_B(B/IB,B/JB)$ is a finitely generated right $R/IB$-module for all prime ideals $J\lhd C$ which contain $I$. As mentioned before, since we are dealing with non-trivial stabilisers, we require some extra notation and a definition.
\begin{Definition}
For two subgroups $H,K$ of a finitely generated abelian group $G$, we say $H$ is {\em complementary} to $K$ if $H\cap K = \{0\}$ and $H\oplus K$ is of finite index in $G$.
\end{Definition}
\begin{remark}
We note that by the classification theorem for finitely generated abelian groups, that a complement always exists and may be chosen to be a free abelian group with $\rank(K) = \rank(G)-\rank(H)$.
\end{remark}
From now on we denote $K \coloneqq \Stab_G(I)$.
\begin{proposition}\label{RIB}
Assume Notation \ref{Not}. Then $R/IB \cong (C/I)\mypound K$ is left and right noetherian.
\end{proposition}
\begin{proof}
We note that for any prime ideal $J\lhd C$,
\begin{equation}
(J:I^{g}) = 
\begin{cases}
              C \textrm{ if } I^{g}\subseteq J, \\
             J \textrm{ else.}
              
\end{cases}\tag{1}
\end{equation}
Then by Lemma \ref{idealiser} , $R = \bigoplus_{g\in G}(I:I^g)g = \left(\bigoplus_{g\in K}Cg\right)\oplus \left(\bigoplus_{g\in G\setminus{K}}Ig\right)$. Hence
$$R/IB = \bigoplus_{g\in G}\frac{(I:I^g)}{I}g = \bigoplus_{k\in K}\frac{C}{I}k.$$
This is a free left $(C/I)$-module which has a basis generated by the elements of $K$. Thus $R/IB = (C/I)\mypound K$.

As $K$ is a subgroup of a finitely generated abelian group it is also finitely generated abelian and $C/I$ is clearly noetherian, $(C/I)\mypound K$ is left and right noetherian by Theorem \ref{MR}.
\end{proof}
\begin{notation}
For a subgroup $H\leq G$ and an ideal $J\lhd C$ we denote specific sets as follows: $$S_{H,V(J)}\coloneqq \{h\in H \mid h.V(J) \subseteq V(I)\},$$
where $h.p$ is the induced action of $G$ on $\spec C$, and
$$T_{H,V(J)} \coloneqq \{h \in H \mid \Tor_1^C(C/I^{-h}, C/J)\neq 0\}.$$
\end{notation}
\begin{lemma}\label{sets}
Assume Notation \ref{Not}. The sets $S_{G,V(J)}$ and $T_{G,V(J)}$ are both $K$-sets.
\end{lemma}
\begin{proof}
We begin with $S_{G,V(J)}$. We must show that if $a \in S_{G,V(J)}$, then $a+K \in S_{G,V(J)}$. Indeed
$$(a+K).V(J) \subseteq K.V(I) \subseteq V(I).$$
Now for $T_{G,V(J)}$ we have 
\begin{align*}\Tor_1^C(C/I^{-a},C/J) = 0 &\iff \Tor_1^C(C/I,C/J^a) = 0 \\
&\iff  \Tor_1^C(C/I^{-k},C/J^a)=0 \quad \forall k \in K \textrm{ as } K = \Stab_{G}(I) \\
&\iff \Tor_1^C(C/I,C/J^{(a+k)})=0  \\
&\iff \Tor_1^C(C/I^{-(a+k)},C/J)=0 \quad \forall k \in K.
\end{align*}
\end{proof}
\begin{theorem}\label{R}
Assume Notation \ref{Not}. Then the following are equivalent:
\begin{enumerate}
    \item[(1)] there exists a subgroup $H\leq G$, complementary to $K$, such that for all points $p \in V(I)$ the set
    $S_{H, p}$
    is finite;
    \item[(2)] $\mathbb{I}_B(IB)$ is right noetherian;
    \item[(3)] for all subgroups $H \leq G$, complementary to $K$, and for all points $p \in V(I)$, $S_{H,p}$ is finite.
\end{enumerate}
\end{theorem}
\begin{proof}
We begin by proving $(1)$ implies $(2)$ using the following claims:
\begin{enumerate}
    \item [(a)]  For $\pi:G \to G/K$, the canonical map, $\pi(S_{G,p})$ is finite for all $p\in V(I)$;
    \item [(b)] $\pi(S_{G, V(J)})$ is finite for all prime ideals $J \lhd C$ where $I \subseteq J$;
    \item [(c)] 
    $\frac{(JB:IB)}{JB}$ is a finitely generated right $(C/I)\mypound K$-module for all prime ideals $J\lhd C$ where $I\subseteq J$.
\end{enumerate}
We proceed by showing $(1)\implies (\mathrm{a}) \implies (\mathrm{b}) \implies (\mathrm{c}) \implies (2)$.

Let us start with $(1)\implies (\mathrm{a})$. Assume for some complement $H\leq G$ of $K$ that $S_{H,p}$ is finite for all $p \in V(I)$. We note that this implies that $S_{H,p}$ is finite for all $p \in \spec C$; indeed, if $S_{H,p}$ were infinite for some $p$, then for any $a \in S_{H,p}$, $S_{H, a.p}$ is infinite and $a.p \in V(I)$.

Now $H+K = H \oplus K$ has finite index, say $m$, in $G$ and let $a_1,\dots, a_m \in G$ be coset representatives. Then $G = \bigsqcup_{i=1}^m(a_i+H+K)$ and so for $p \in V(I)$,
\begin{align*}
    S_{G,p}&=\{\alpha \in G \mid \alpha . p \in V(I)\}\\
    &=\bigsqcup_{i=1}^m\{\alpha \in a_i+H+K \mid \alpha . p \in V(I)\}\\
    &=\bigsqcup_{i=1}^m\left(\{\alpha \in a_i+H \mid \alpha . p \in V(I)\}+K\right)  \quad \quad \textrm{as } K= \Stab_{G}(I)\\
    &= \bigsqcup_{i=1}^m \left(a_i+\{h \in H \mid (h+a_i) . p \in V(I)\}+K\right) \\
    &=\bigsqcup_{i=1}^m(a_i+S_{H,a_i.p}+K).
\end{align*}
Hence, as $S_{H,a_i.p}$ is finite by assumption, $\pi(S_{G, p})$ is finite as required.

For $(\mathrm{a}) \implies (\mathrm{b})$,

let $J \lhd C$ be a prime ideal such that $I \subseteq J$. Since
$S_{G,V(J)} = \bigcap_{p\in V(J)} S_{G, p}$ the result follows and, as $S_{G,V(J)}$ is a $K$-set by Lemma \ref{sets}, $S_{G,V(J)}$ is a finite union of cosets of $K$.

Now for $(\textrm{b}) \implies(\textrm{c})$.

Again, let $J\lhd C$ be prime such that $I\subseteq J$. We have $\frac{(JB:IB)}{JB}= \bigoplus_{g \in G}\frac{(J:I^{g})}{J}g$ by Lemma \ref{idealiser}.
Then 
\begin{align*}
    \frac{(JB:IB)}{JB}&= \bigoplus_{g\in G}\frac{(J:I^{g})}{J}g = \bigoplus_{g\in S_{G,V(J)}}\frac{C}{J}g \quad \textrm{ by Equation $(1)$ }  \\
    &= \bigoplus_{j=1}^p \bigoplus_{k\in K} \frac{C}{J}(b_j+k)
\end{align*}
where  the $b_j\in G$ are the coset representatives for the finite set $\pi(S_{G,V(J)})$.
We note that $(C/J)b_j$ is a right $C/I$-module. We must check that $I$ acts trivially on $(C/J)b_j$. Since $C$ acts on $(C/J)b_j$ by
$\bar{f}b_jc = \bar{f}c^{b_j}b_j$ and $b_j\in S_{G,V(J)}$, $I^{b_j}\subseteq J$ as required. Hence $(C/J)b_jI =0$ and $$\oplus_j\oplus_k(C/J)(b_j+k) \cong \left(\oplus_{j} (C/J)(b_j+k)\right)\otimes_{C/I}(\oplus_k C/Ik) = \oplus_j \frac{C}{J} b_j \otimes_{C/I} ((C/I) \mypound K),$$
which is a finitely generated right $(C/I)\mypound K$-module.

Finally, we must prove that $(\mathrm{c})\implies (2)$. By Theorem  \ref{STPrime}, $R$ is right noetherian if and only if $\bigoplus_{g\in G}\frac{(P:I^g)}{P}g$ is a finitely generated right $R/IB$-module for all prime $J\lhd C$ such that $I\subseteq J$. By Proposition \ref{RIB}, $R/IB \cong (C/I)\mypound K$ is left and right noetherian, and so $R$ is right noetherian.

We move onto $(2)$ implies $(3)$. Suppose that there exists a complement $H\leq G$ to $K$ such that for some $p \in V(I)$,
$$S_{H,p} = \{h \in H \mid h.p \in V(I)\}$$ is infinite. Consider $M = I(p)$, the ideal of $C$ associated to $p \in V(I)$. We note, if $h_1, h_2 \in S_{H,p}$ are distinct, then $h_1+K \neq h_2+K$ as $H\cap K = \{0\}$.

Then
\begin{align*}
    \bigoplus_{g\in G}\frac{(M:I^{g})}{M}g &=  \bigoplus_{g\in S_{G,p}}\frac{C}{M}g \geq\bigoplus_{g\in S_{H,p}+K}\frac{C}{M}g\\
     &= \bigoplus_{\substack{a\in S_{H,p}\\k\in K}}\frac{C}{M}(a+k)\\
     &= \bigoplus_{a\in S_{H,p}}\frac{C}{M}a\otimes_{C/I}\bigoplus_{k\in K}\frac{C}{I}k\\
     &= \bigoplus_{a\in S_{H,p}}\frac{C}{M}a\otimes_{C/I}((C/I)\mypound K).
\end{align*}
By assumption, this is not a finitely generated right module over $(C/I)\mypound K \cong R/IB$-module, thus $R$ is not right noetherian by Theorem  \ref{STPrime}.

The implication $(3)$ implies $(1)$ is trivial. This completes the proof.
\end{proof}
\subsection{Left noetherianity}
Now we turn our attention to left noetherianity of our idealiser $R$. By Theorem  \ref{tor}, we must show that $R/IB \cong (C/I)\mypound K $ is left graded-noetherian and $\bigoplus_{g \in G}\Tor_1^C(C/I,C/P^g)g$ is a noetherian left $R$-module (equivalently, $R/IB$-module) for all prime ideals $P\lhd C$. We note that we may not restrict to only prime ideals which contain the ideal $I$. This is one aspect of the left structure of idealisers which is more complicated and highlights that idealisers are rings which can have different left and right structures.

By Proposition \ref{RIB}, $R/IB$ is left noetherian. We now characterise left noetherianity of $R = \mathbb{I}_B(IB)$.
\begin{theorem}\label{L}
Assume Notation \ref{Not}. Then the following are equivalent:
\begin{enumerate}
    \item[(1)] There exists a subgroup $H\leq G$, complementary to $K$, such that 
    $T_{H,V(J)} = \{a \in H \mid \Tor_1^C(C/I^{-a}, C/J)\neq 0\}$
    is finite for all prime ideals $J\lhd C$;
    \item[(2)] $R$ is left noetherian;
    \item[(3)] For all complementary subgroups $H$ to $K$, $T_{H,V(J)}$ is finite for all prime ideals $J \lhd C$.
\end{enumerate}
\end{theorem}
\begin{proof}
We aim to show that $(1)$ implies that $\pi\left(T_{G,V(J)}\right)$ is finite. We note that $\{a \in H \mid \Tor_1^C(C/I^{-a},C/J)\neq 0\} = \{a \in H \mid \Tor_1^C(C/I,C/J^a)\neq 0\}$ and we use the latter in the proof for ease of notation. Again, $H\oplus K$ has finite index in $G$, say $m$, with coset representatives $a_1,\dots a_m$. Then we have, for $J\lhd C$,
\begin{align*}
    T_{G, V(J)} &= \{\alpha \in G \mid \Tor_1^C(C/I, C/J^{\alpha})\neq 0\} \\
                 &= \bigsqcup_{i=1}^m\{\alpha \in a_i+H+K \mid \Tor_1^C(C/I, C/J^{\alpha})\neq 0 \} \\
                 &= \bigsqcup_{i=1}^m \left(a_i+\{\alpha \in H \mid \Tor_1^C(C/I, C/J^{(a_i+\alpha)})\neq 0 \}+K\right) \textrm{ as in Lemma \ref{sets}}\\
                 &= \bigsqcup_{i=1}^m \left(a_i+T_{H,V(J^{a_i})}+K\right).
\end{align*}
Since $T_{H,V(J^{a_i})}$ is finite for $i=1,\dots,m$, we obtain that $\pi\left(T_{G,V(J)}\right)$ is finite. As $\pi\left(T_{G,V(J)}\right)$ is a right $K$-set by Lemma \ref{sets}, $T_{G,V(J)}$ is a finite union of cosets of $K$.

Then
\begin{align*}
    \Tor_1^B(B/IB, B/BJ) &= \frac{IB\cap BJ}{IBJ} =\bigoplus_{g \in G} \frac{Ig\cap gJ}{IgJ}\\
    &=\bigoplus_{g\in G} \frac{Ig\cap J^{g}g}{IJ^gg} = \bigoplus_{g\in G} \frac{I\cap J^{g}}{IJ^g}g \\
    &=\bigoplus_{g\in G} \Tor_1^C(C/I, C/J^g)g\\
                         &= \bigoplus_{g\in T_{G,V(J)}} \Tor_1^C(C/I, C/J^g)g\\
                         &= \bigoplus_{j=1}^m \bigoplus_{k \in K} \Tor_1^C(C/I, C/J^{b_j+k})(k+b_j)
\end{align*}
where the $b_j\in G$ are the coset representatives for the finite set $\pi\left(T_{G,V(J)}\right)$.

Observe that $\bigoplus_{k \in K} \Tor_1^C(C/I, C/J^{b_j+k})(b_j+k)$ is a finitely generated left $(C/I)\mypound K$-module. Indeed, using the $K$-invariance of $I$, both $\bigoplus_{k \in K}I\cap J^{b_j+k}k$ and $\bigoplus_{k\in K}IJ^{b_j+k}k$ are graded left ideals of $C\mypound K$, and hence are finitely generated left $C\mypound K$-modules. Thus $\bigoplus_{k\in K}\Tor_1^C(C/I,C/J^{b_j+k})k = \bigoplus_{k\in K}\frac{I\cap J^{b_j+k}}{IJ^{b_j+k}}k$ is a finitely generated left $C\mypound K$-module. As $I$ acts trivially, $\bigoplus_{k\in K}\Tor_1^C(C/I,C/J^{b_j+k})k$ is also a finitely generated left $(C/I)\mypound K$-module. Thus $\Tor_1^B(B/IB, B/BJ)$ is a finitely generated left $(C/I)\mypound K$-module. As $J$ was arbitrary, by Theorem  \ref{tor},  $R$ is left noetherian.

We now show $(2)\implies (3)$. Suppose there exists a complementary subgroup $H\leq G$ to $K$ such that for some prime ideal $J \lhd C$, $T_{H,V(J)}$ is infinite.

Then
\begin{align*}
\Tor_1^B(B/IB, B/BJ)&=\bigoplus_{g \in G}\Tor_1^C(C/I,C/J^g)g = \bigoplus_{g\in T_{G,V(J)}}\Tor_1^C(C/I,C/J^g) g\\
&\geq \bigoplus_{g \in T_{H,V(J)}\oplus K}\Tor_1^C(C/I,C/J^g)g\\
&=\bigoplus_{a \in T_{H,J}}\bigoplus_{k\in K}\Tor_1^C(C/I,C/J^{a+k})(a+k),
\end{align*}
which is an infinite direct sum of non-zero $(C/I)\mypound K$-modules. Hence $R$ is not left noetherian as required.
\end{proof}

To end this section, we apply our results in the case that $I$ is a maximal ideal of $C$. By Theorem \ref{R}, $R$ is right noetherian if and only if there exists a complementary subgroup $H\leq G$ to $\Stab_{G}(I)$ such that
$$S_{H,p}=\{h\in H\mid h.p=p\}$$
is finite where $p=V(I)$. But $S_{H,p} = H\cap \Stab_G(I) = 0$, so this always holds. Hence $R$ is always right noetherian. We note that this result may be considered a graded version of Theorem $1.1$.

Let us now consider left noetherianity. We require for some complementary subgroup $H\leq G$
$$\{h\in H\mid \Tor_1^C(C/I^{-h},C/J)\neq 0\}$$
to be finite for all prime $J\lhd C$. Since $\Tor_1^C(C/I^{-h},C/J)$ is supported on $V((-h).I)\cap V(J)$ the only way that $\Tor_1^C(C/I^{-h},C/J)\neq 0$ is if $J^h\subseteq I$ or equivalently $h.p=V(I)\in V(J)$. That is to say, for all subvarieties, $Z$, of $\spec C$, $\{h.p\}_{h \in H}\cap Z$ is finite. This property already exists in the literature and is known as {\em critical density}.

\begin{Definition}
Let $X$ be an affine variety and let $S$ be an infinite subset of $X$. We say $S$ is {\em critically dense} if $S\cap Y$ is finite for all subvarieties $Y$ of $X$. Equivalently, any infinite subset of $X$ is Zariski dense.
\end{Definition}

We summarise the preceding discussion in the following theorem.
\begin{theorem}
Assume Notation \ref{Not}. Suppose that $I$ is a maximal ideal of $C$ and let $p=V(I)$. Then $R$ is always right noetherian. Further, $R$ is left noetherian if and only if $G.p$ is critically dense.
\qed
\end{theorem}
\subsection{Critical Transversality}
In the previous two sections we have found conditions on for $\mathbb{I}_B(IB)$ to be right or left noetherian which we note are very different. For the right-hand side we have a condition which is based on the orbit of points in $V(I)$. However, on the other side the condition for left noetherianity is much less clear. In this section we show that this condition has a geometric analogue as Sierra showed for twisted homogeneous coordinate rings in \cite{Sierra} and we prove Theorem \ref{T2}.

We now show that the condition of left noetherianity is closely related to the notion of critical transversality as defined in the introduction. 
\begin{lemma}\label{CTLN}
Assume Notation \ref{Not}. Then the following are equivalent:
\begin{enumerate}
    \item [(1)] For all prime ideals $J\lhd C$, the set 
    $$\{h \in H\mid \Tor_1^C(C/I^{-h},C/J)\neq 0\}$$
    is finite.
    \item[(2)] For all prime ideals $J\lhd C$, the set
    $$A(J) = \{h \in H \mid V(I^{-h}) \textrm{ is not homologically transverse to } V(J)\}$$
    is finite.
\end{enumerate}
\end{lemma}
Before beginning the proof, we establish some terminology which we will need if $\spec C$ is singular.
\begin{Definition}
Let $X = \spec C$ be an affine variety. We define the {\em singular stratification} of $X$ iteratively as follows:

define $X^{(1)} = \{m \in \mspec(C) \mid \gdim_C\left(C_{(m)}\right) = \infty\}$ to be the singular locus of $X$, which is closed, and define $X^{(n)}$ to be the singular locus of $X^{(n-1)}$. The singular stratification is preserved under automorphisms.
\end{Definition}
We have the following Lemma, originally due to Mel Hochster.
\begin{lemma}\cite[Lemma 5.3]{Sierra}\label{PD}
Suppose that $V(I)$ is homologically transverse to all parts of the singular stratification of $\spec C$. Then
$$\pdim_C(C/I)<\infty.$$
\end{lemma}
\begin{proof}[Proof of Lemma \ref{CTLN}]
That $(2)$ implies $(1)$ is trivial. 

Assume $(1)$. We may assume that $H$ is infinite. We note that 
$$A(J) = \bigcup_{j\geq 1}\{h \in H\mid \Tor_j^C(C/I^{-h},C/J)\neq 0\}.$$
We note that whilst standard homological arguments imply that 
\begin{align*}A_j(J) = \{h \in H\mid \Tor_j^C(C/I^{-h},C/J)\neq 0\}\end{align*} is finite for each $j\geq 1$, this is not enough to conclude $(2)$. However, if the projective dimension of $C/I$ is finite, $(2)$ immediately follows.

We first claim that for any finitely generated $C$-module $M$ and $j\geq 1$, the set 
$$\{h \in H \mid \Tor_j^C(C/I^{-h},M)\neq 0\}$$
is finite. We induct on $j\geq 1$. Firstly, for $j=1$. By \cite[Proposition 3.7]{Eis}, $M$ has a filtration 
$$0 = M_0\subset M_1 \subset \dots \subset M_n=M$$
with each $M_{i+1}/M_i \cong C/P_i$ for some prime ideal $P_i\lhd C$. A similar argument to Theorem \ref{tor}, inducting on $n$ gives that $\{h \in H \mid \Tor_1^C(C/I^{-h},M)\neq 0\}$ is finite. Now let $j>1$. We may construct a short exact sequence
$$0 \to K \to C^m \to M \to 0$$
where $K$ is a finitely generated $C$-module. From the long exact sequence in Tor we obtain
$$\Tor_j^C(C/I^{-h},M) \cong \Tor_{j-1}^C(C/I^{-h},K),$$
from which the claim follows by induction as the right-hand side vanishes for all but finitely many $h\in H$.

From the claim, $V(I)$ is homologically transverse to all $H$-invariant subvarieties of $\spec C$. This is because, if $J$ is $H$-invariant, $A_{j}(J)$ is either all of $H$ or trivial and, as $\{h \in H\mid \Tor_1^C(C/I^{-h},C/J)\neq 0\}$ is finite, it must be the latter. In particular, $V(I)$ is homologically transverse to the singular stratification of $\spec C$ as each of the terms in the singular stratification corresponds to a factor $C/A$ where $A$ is an $H$-invariant ideal of $C$. Hence $C/I$ has finite projective dimension by Lemma \ref{PD} and $A_j(J) = 0$ for all $j > \pdim_C C/I$. Thus $A(J) = \bigcup_{j\geq 1} A_j(J)$ is finite as required.
\end{proof}
The following Theorem now follows as a corollary of Lemma \ref{CTLN}.
\begin{theorem}[Theorem \ref{T2}]\label{finalLeft}
Assume Notation \ref{Not}. Then $R$ is left noetherian if and only if, for some complement $H\leq G$ of $K$, $\{I^{-h}\}_{h \in H}$ is critically transverse.\qed
\end{theorem}
As mentioned in the introduction, results obtained by Sierra for idealisers in twisted homogeneous co-ordinate rings are similar to our results in the case that $G = \ZZ$.
\begin{theorem}\cite[Theorem 10.2]{Sierra}\label{notvague}
Let $X$ be a projective variety, let $\sigma \in \Aut_{\kk} X$, let $\mathcal{L}$ be a $\sigma$-ample invertible sheaf on $X$, and let $Z$ be an irreducible, closed subscheme of $X$ of infinite order under $\sigma$. Let $B = (X,\mathcal{L},\sigma)$ be a twisted homogeneous coordinate ring and let $I$ be the right ideal of $B$ corresponding to $Z$. 

If for all $p\in Z$, the set $\{n\geq 0 \mid \sigma^n(p)\in Z\}$ is finite then $\mathbb{I}_B(I)$ is right noetherian. 
If the set $\{\sigma^nZ\}_{n \in \ZZ}$ is critically transverse, then $\mathbb{I}_B(I)$ is left noetherian.

\end{theorem}

We note that both of the sets involved correspond with those we obtained in Theorems \ref{R} and \ref{L} in the case that the stabiliser is trivial and $G = \ZZ$.

\section{Idealisers defined by subvarieties of the plane}
Now that we have an abstract set of conditions for when idealiser rings are left and right noetherian, let us see how these work in practice. We consider $C = \CC[x,y]$ and $\ZZ^2\subseteq \CC^2$ acting by translation. We will see that the noetherianity of the idealisers we obtain depends on subtle arithmetic results about integer points on varieties.
\begin{example}\label{EX}
Let us consider $C=\CC[x,y]$, and $\sigma, \tau \in \Aut_{\CC}(\CC[x,y])$ defined by 
\begin{align*}
    &\sigma(x)=x+1,\quad \sigma(y)=y,\\
    &\tau(x)=x,\quad \quad \ \ \tau(y)=y+1.
\end{align*}
Define $B = C\mypound \ZZ^2$ with $\ZZ^2 = (\sigma,\tau)$ acting on $C$ by translation.
By Bezout's theorem, the only irreducible curves with non-trivial stabiliser under this action are lines with rational slope (including slope $\infty$). We shall treat this case separately.

If $V(I)$ is not such a line then $I\lhd C$ is prime with trivial stabiliser. By Theorem \ref{R}, $\mathbb{I}_B(IB)$ is right noetherian if and only if
$$\{a\in \ZZ^2 \mid a.p\in V(I)\}$$
is finite for all $p\in V(I)$. This is equivalent to
$$V(I)\cap (b+\ZZ^2)$$
being finite for all $b\in \CC^2$.

Further, $\mathbb{I}_B(IB)$ is left noetherian if and only if 
$$\{a \in \ZZ^2 \mid \Tor_1^C(C/I^{-a},C/J)\neq 0\}$$
is finite for all prime ideals $J\lhd C$.

Our task is to understand these geometric conditions better. We shall split our work into two cases: when $I$ is and is not maximal. Firstly, when $I$ is maximal, then $\Stab_\ZZ^2(I)$ is trivial and we claim that $\mathbb{I}_B(IB)$ right but it is not left noetherian. Indeed, if $I$ is maximal then $\mathbb{I}_B(IB)$ is easily shown to be right noetherian by Theorem \ref{STPrime}. Also, as $V(I) = p$ for some point $p=(p_1,p_2) \in \mathbb{A}^2$, let us consider the subvariety $Y = V(x-p_1)\subseteq \mathbb{A}^2$. As $\tau^{i}(p)\in  Y$ for all $i \in \ZZ$, $\ZZ^2.V(I)$ cannot be critically dense and thus, by Theorem $3.15$, $\mathbb{I}_B(IB)$ is not left noetherian.
We sumarise this in the following Proposition.
\begin{proposition}
Assume the setup from Example \ref{EX}. Suppose $I$ is a maximal ideal. Then $\mathbb{I}_B(IB)$ is right but not left noetherian.\qed
\end{proposition}

Now we consider what happens when $I$ is prime but not maximal, then $I=(f)$ for some irreducible polynomial $f \in \CC[x,y]$ and corresponds to a plane curve in $\mathbb{A}^2$. We now view $\mathbb{I}_B(IB)$ as a subring of the skew field of fractions associated to $B$, $Q(B)$, in which $f$ is invertible. Then the conjugation action by $f$:
\begin{align*}
    \theta \colon \ &Q(B) \to Q(B),\\ 
    &\theta(a) = f^{-1}af
\end{align*}
is an isomorphism of $Q(B)$. We note that $\mathbb{I}_B(IB) = C+fB$ and so $\theta(\mathbb{I}_B(IB)) = C+Bf = \mathbb{I}_B(BI)$ as $f$ commutes with $C$. Thus $\theta$ restricted to $\mathbb{I}_B(IB)$ induces an isomorphism between $\mathbb{I}_B(IB)$ and $\mathbb{I}_B(BI)$. Hence $\mathbb{I}_B(IB)$ is left noetherian if and only if $\mathbb{I}_B(BI)$ is left noetherian which, by the opposite-sided version of Theorem \ref{R}, happens if and only if $V(I)\cap (b+\ZZ^2)$ is finite for all $b\in \CC^2$. That is to say, $\mathbb{I}_B(IB)$ is left noetherian if and only if it is right noetherian. We note that if $I = (f)$ then $\Tor_1^C(C/I,C/J)=\neq 0$ if and only if $J\supseteq I$. So the Tor condition also gives that $\mathbb{I}_B(IB)$ is left noetherian if and only if $V(I)\cap (b+\ZZ^2)$ is finite. We summarise the above discussion in the following Proposition.
\begin{proposition}\label{ints}
Assume the setup from Example \ref{EX}. Suppose $I$ is a non-maximal prime ideal that does not correspond to a line of rational slope. Then $\mathbb{I}_B(IB)$ is left and right noetherian if and only if the set
$$V(I)\cap(b+\ZZ^2)$$
is finite for all $b\in \CC^2$.\qed
\end{proposition}

We now seek to understand which ideals $I$, or indeed which curves $V(I)$, satisfy this condition. 
We turn to Siegel's theorem on integral points.
\begin{theorem}\cite[Theorem 3.2]{Zan}
Let $C$ be an affine irreducible curve over a number field $k$, and
suppose it has infinitely many integral points. Then $C$ has genus $0$ and at most
two points at infinity.
\end{theorem}
We observe that if we were simply interested in whether $V(I)\cap \ZZ^2$ was finite where $V(I)$ was defined over some finite field extension of $\QQ$ then this would be a straightforward application of this theorem. However, this is not the case and in addition we are interested in (possibly non-rational) translations of $V(I)$ and hence we search for some `dynamical' Siegel's Theorem. Thus we turn to Lang's generalisation of Siegel's Theorem from \cite[Chapter VII pp.121 and  Theorem $4$]{Lang}:
\begin{theorem}\label{Lang}
If $C$ is an affine curve defined over a ring $S$ finitely generated over $\ZZ$, and if its genus is $g\geq 1$, then $C$ has only a finite number of points in $S$.
\end{theorem}
We note that if $g\geq 2$ this follows from Falting's Theorem. 

Applying the results, we obtain:
\begin{theorem}
With the setup from Example \ref{EX}, let $X=V(I)$ be an irreducible curve. If $X$ has genus $\geq 1$, then $\mathbb{I}_B(IB)$ is left and right noetherian.
\end{theorem}
\begin{proof}
Let $I = (f(x,y))$ for some irreducible polynomial $f(x,y)\in \mathbb{C}[x,y]$ and suppose that $\mathbb{I}_B(IB)$ is not noetherian. Then, by Proposition \ref{ints} there exists $c,d \in \CC$ such that $g\coloneqq f(c+x,d+y)$ has infinitely many integer solutions. Let $S$ be the $\ZZ$-algebra generated by the coefficients of $g$. Then $g$ defines a curve over $S$ and $g$ has an infinite number of solutions in $S$ (from its infinite integer solutions). Thus, by Theorem \ref{Lang}, $g$ must have genus $0$. Hence, $f$ also has genus $0$ as required.
\end{proof}
For example, if $X$ were a smooth cubic curve then the idealiser associated to this curve will always be left and right noetherian. So any cubic curve with non-trivial $j$-invariant will give a left and right noetherian idealiser.

However, if we consider a genus $0$ curve of the form 
$$x^2-ny^2=1$$
where $n$ is a given positive nonsquare integer, then it is a result of Lagrange that this curve has an infinite number of integral points. As an example, consider $x^2-7y^2=1$, then the integer solutions $(x_{k+1},y_{k+1})$ are given by the recurrence formula:
\begin{align*}
    &x_{k+1} = x_1x_k+7y_1y_k \\
    &y_{k+1} = x_1y_k+y_1x_k
\end{align*}
where $(x_1, y_1) = (8,3)$. So $\mathbb{I}_B((x^2-7y^2-1)B)$ is neither right nor left noetherian.

Recall that we did not consider lines with rational slope as they do not have trivial stabiliser. We deal these lines now.
\begin{proposition}
Assume the setup from Example \ref{EX}. Let $I\lhd C$ be a prime ideal in $C$ corresponding to a line in $\mathbb{A}^2$. Then $\mathbb{I}_B(IB)$ is right and left noetherian.
\end{proposition}
\begin{proof}
If $I$ corresponds to a line of irrational slope, $\Stab_{\ZZ^2}(I)=0$ and $V(I)\cap(b+\ZZ^2)<\infty$ for all $b \in \CC^2$, so by Proposition \ref{ints}, $\mathbb{I}_B(IB)$ is both left and right noetherian.

Now $I = (f) = (mx-ny-p)$ where $m,n,p \in \CC$. Without loss of generality $m,n \in \ZZ$ as the slope is rational and, by symmetry, we may assume $n\neq 0$. Then, as $\sigma^i\tau^j(f) = f$ if and only if $(i,j) \in (n,m)\ZZ$, $K = \Stab_{\ZZ^2}(I) = (n,m)\ZZ$. Consider $H = (0,1)\ZZ$. As $n \neq 0$, $H\cap K = \{0\}$ and, since $H,K \cong \ZZ$, $H \oplus K = \ZZ^2$ so $H$ is a complement to $K$. Further, as $\tau^j(p)\notin V(I)$ for any $p \in V(I)$ and $j\neq 0$, the set $S_{H,p} = \{(0,0)\}$ and hence, by Theorem \ref{R}, $\mathbb{I}_B(IB)$ is right noetherian. For left noetherianity we must show that the set $\{(m,n) \in (0,1)\ZZ \mid I^h \textrm{ is not homologically transverse to } J\}$ is finite for all prime $J\lhd C$. Then, as $V(I^h) \subseteq \mathbb{A}^2$ is a plane curve, $I^h$ and $J$ can only have a non homologically transverse intersection if $I^h \subseteq J$. But this can only happen at most once when $H = (0,1)\ZZ$, hence the set is finite and $\mathbb{I}_B(IB)$ is left noetherian.
\end{proof}
\end{example}
\section{More general groups}
In this section we briefly consider noetherianity of idealisers in group rings with polycyclic-by-finite groups. Providing a complement actually exists - something which may not happen in non-abelian groups - the same form of argument as in Theorems \ref{R} and \ref{L} goes through. We begin by showing that the sets 
$$S_{G,J}=\{g\in G \mid I^g \subseteq J\}$$
and
$$T_{G,J} = \{g \in G \mid \Tor_1^C(C/I,C/J^g)\neq 0\}$$
are right and left $K$-sets respectively.

\begin{lemma}\label{LRK}
The sets 
$$S_{G,J}=\{g\in G \mid I^g \subseteq J\}$$
and
$$T_{G,J} = \{g \in G \mid \Tor_1^C(C/I^{g^{-1}},C/J)\neq 0\}$$
are right and left $K$-sets respectively.
\end{lemma}
\begin{proof}First for $S_{G,J}$ where $J \supseteq I$. Let $a \in S_{G,J}$ and let $k \in K$.
Then $$I^{ak}\subseteq I^a \subseteq J$$
by the definition of $K$.

Now for $T_{G,J}$ where $J\lhd C$ is an arbitrary prime ideal. Then
\begin{align*}
\Tor_1^C(C/I^{a^{-1}},C/J) =0 & \iff 
\Tor_1^C(C/I,C/J^a) =0 \iff  \Tor_1^C(C/I^{k^{-1}},C/J^a)=0 \\ &\iff  \Tor_1^C(C/I,C/J^{ka})=0 \iff
\Tor_1^C(C/I^{(ka)^{-1}},C/J)=0
\end{align*}
as required. We note that neither of these sets are necessarily two-sided $K$-sets.
\end{proof}
Armed with this result, the generalisation of Theorems \ref{R} and \ref{L} follows through exactly the same argument.
\begin{theorem}\label{R'}
Assume Notation \ref{Not} and that $K = \Stab_G(I)$ has at least one complementary subgroup. Then the following are equivalent:
\begin{enumerate}
    \item[(1)] there exists a subgroup $H\leq G$, complementary to $K$, such that for all points $p \in \spec C$ the set
    $S_{H, p}$
    is finite;
    \item[(2)] $R$ is right noetherian;
    \item[(3)] for all subgroups $H \leq G$, complementary to $K$, and for all points $p \in \spec C$, $S_{H,p}$ is finite.
\end{enumerate}
\end{theorem}
We note that the first condition in Theorem \ref{R'} is for all points $p\in \spec C$ as opposed to just points in $V(I)$. This is because $G$ no longer being abelian means we cannot only consider $S_{H,p}$ for $p\in V(I)$.
\begin{theorem}\label{L'}
Assume Notation \ref{Not} and that $K$ has at least one complementary subgroup. Then the following are equivalent:
\begin{enumerate}
    \item[(1)] there exists a subgroup $H\leq G$, complementary to $K$, such that 
    $$T_{H,V(J)} = \{a \in H \mid \Tor_1^C(C/I^{-a}, C/J)\neq 0\}$$
    is finite for all prime ideals $J\lhd C$;
    \item[(2)] $R$ is left noetherian;
    \item[(3)] for all complementary subgroups $H$ to $K$, $T_{H,V(J)}$ is finite for all prime ideals $J \lhd C$.
\end{enumerate}
\end{theorem}
To close, we consider an example where the stabiliser does not have a complementary subgroup.
\begin{example}
We consider $G = \begin{pmatrix}
1&\ZZ&\ZZ \\
0 & 1 & \ZZ \\
0 & 0 & 1
\end{pmatrix}$
the Heisenberg group with centre $Z(G)=\begin{pmatrix}
1&0&\ZZ \\
0 & 1 & 0 \\
0 & 0 & 1
\end{pmatrix}$. 
We have that $G/Z(G) \cong \ZZ^2$, which we let act on $\CC[x,y]$ as in Example \ref{EX}. Requiring that $Z(G)$ acts trivially gives an induced action of $G$ on $\CC[x,y]$ with each maximal ideal having $Z(G)$ as its stabiliser. We claim that the centre does not have a complementary subgroup. Indeed, to obtain $[G:HZ(G)]< \infty $, since $Z(G) = [G,G]$ and the Hirsch length of $G$ is $3$, $H$ would have to be free abelian of rank $2$. But no subgroup of $G$ of this form intersects the centre trivially. Hence $Z(G)$ is an example of of a subgroup with no complement. Let $B = \mathbb{C}[x,y]\mypound G$, $M\lhd \CC[x,y]$ a maximal ideal, and $R= \mathbb{I}_B(MB)$.

However, we can still determine whether these idealisers are left or right noetherian. Indeed, as $M$ is a maximal ideal, we only have two choices for $J\supseteq M$ in $(JB:MB)/JB$, namely $J = M$ and $J = C$. In either case, $(JB:MB)/JB$ is a finitely generated right $(C/I)\mypound K$-module and hence, by Theorem \ref{STPrime}, $R$ is right noetherian. For left noetherianity, as $M$ is a maximal ideal of $\mathbb{C}[x,y]$ it is of the form $M = (x-p,y-q)$ for some $p,q \in \mathbb{C}$, consider the ideal $J = (x-p)\lhd \mathbb{C}[x,y]$. Let $H = \begin{pmatrix}
1&0&0 \\
0 & 1 & \ZZ \\
0 & 0 & 1
\end{pmatrix}$, then $J^h= J\subseteq M$ for all $h \in H$ and so $\Tor_1^C(C/M,C/J^h) =\Tor_1^C(C/M,C/J)\neq 0$. If $h,h' \in H$ are distinct then $hK\neq h'K$, so $\bigoplus_{h \in H}\Tor_1^C(C/M,C/J^h)h$ cannot be a finitely generated left $(C/I)\mypound K$-module. Thus, since $\bigoplus_{h \in H}\Tor_1^C(C/M,C/J^h)h\leq \bigoplus_{g \in G}\Tor_1^C(C/M,C/J^g)g $, $\bigoplus_{g \in G}\Tor_1^C(C/M,C/J^g)g$ cannot be finitely generated as a left $(C/I)\mypound K$-module and hence, by Theorem \ref{tor}, $R$ is not left noetherian.
\end{example}

\providecommand{\bysame}{\leavevmode\hbox to3em{\hrulefill}\thinspace}
\providecommand{\MR}{\relax\ifhmode\unskip\space\fi MR }
\providecommand{\MRhref}[2]{%
  \href{http://www.ams.org/mathscinet-getitem?mr=#1}{#2}
}
\providecommand{\href}[2]{#2}

\end{document}